\documentclass[12pt, reqno]{amsart}
\usepackage[utf8]{inputenc}
\usepackage{amssymb,mathtools,cite,enumerate,color,eqnarray,hyperref,amsfonts,amsmath,amsthm,setspace,tikz,verbatim,charter}

\usepackage[a4paper,margin=1.9cm,top=2.5cm,bottom=2.5cm,centering,vcentering]{geometry}
\definecolor{ao(english)}{rgb}{0.0, 0.5, 0.0}
\definecolor{pink}{rgb}{1.0, 0.0, 1.0}

\hypersetup{colorlinks=true, linkcolor=ao(english),citecolor=ao(english)}

\theoremstyle{plain}
\newtheorem{theorem}{Theorem}
\numberwithin{theorem}{section}

\newtheorem*{corollary*}{Corollary}
\newtheorem*{Example*}{Example}

\newtheorem{lemma}[theorem]{Lemma}

\newtheorem{conjecture}[theorem]{Conjecture}
\theoremstyle{definition}
\newtheorem{definition}[theorem]{Definition}
\newtheorem*{def*}{Definition}
\newtheorem*{theorem*}{Theorem}
\newtheorem*{example}{Example}

\newtheorem*{definition*}{Definition}

\numberwithin{equation}{section}

\begin{document}

\title[New Arithmetic Properties for Overpartitions Where Nonoverlined Parts Are $\ell$-Regular]{New Arithmetic Properties for Overpartitions where Nonoverlined Parts are $\ell$-Regular}

\author[H. Nath]{Hemjyoti Nath}
\address[H. Nath]{Department of Mathematics, University of Florida, P.O. Box 118105, Gainesville, FL 32611-8105, USA}
\email{h.nath@ufl.edu}

\author[M. P. Saikia]{Manjil P. Saikia}
\address[M. P. Saikia]{Mathematical and Physical Sciences division, School of Arts and Sciences, Ahmedabad University, Ahmedabad 380009, Gujarat, India}
\email{manjil@saikia.in}

\author[J. A. Sellers]{James A. Sellers}
\address[J. A. Sellers]{Department of Mathematics and Statistics, University of Minnesota Duluth, Duluth, MN 55812, USA}
\email{jsellers@d.umn.edu}

\thanks{The second author is partially supported by a Start-Up Grant from Ahmedabad University, India (Reference No. URBSASI24A5).}



\begin{abstract}
In this paper, we study the partition functions $\overline{R_\ell^\ast}(n)$, which count the number of overpartitions of $n$ where the non-overlined parts are $\ell$-regular for a given $\ell$. Using elementary techniques, as well as the theory of modular forms, we establish several new arithmetic properties, including infinite families of congruences for these functions.
\end{abstract}

\maketitle

\section{Introduction}

An integer partition $\lambda$ of $n$ is a non-increasing sequence of positive integers \[\lambda =(\lambda_1, \lambda_2, \ldots, \lambda_k),\] such that $\lambda_i\geq \lambda_{i+1}$ for $1\leq i \leq k-1$ and $\sum\limits_{i=1}^k\lambda_i=n$. For instance, the $3$ partitions of $3$ are $(3)$, $(2,1)$ and $(1,1,1)$. We denote by $p(n)$ the number of partitions of $n$, and its generating function is given by
\[
\sum_{n\geq 0}p(n)q^n=\frac{1}{f_1},
\]
where we use the shorthand notation
\[
f_k:=\prod_{i\geq 1}(1-q^{ki}), \quad |q|<1.
\]
Due to their inherent simplicity and beauty, partitions have been studied extensively since the time of Euler, and various generalizations are known for them.

One such generalization is that of overpartitions (studied by Corteel and Lovejoy \cite{over}), which are partitions wherein the first occurrence of a part may be overlined. For instance, the overpartitions of $3$ are
\[
(3), (\overline{3}), (2,1), (\overline{2},1), (2,\overline{1}), (\overline{2}, \overline{1}), (1,1,1), \text{~and~}(\overline{1},1,1). 
\]
We denote the number of overpartitions of $n$ by $\overline{p}(n)$, and their generating function is given by
\[
\sum_{n\geq 0}\overline{p}(n)q^n=\frac{f_2}{f_1^2}.
\]
Numerous generalizations of overpartitions are known and studied. This paper is concerned with one such generalization.

Alanazi, Alenazi, Keith, and Munagi \cite{AlanaziAlenaziKeithMunagi} introduced the partition function $\overline{R_\ell^\ast}(n)$, which counts the number of overpartitions of $n$ where the non-overlined parts are $\ell$-regular (which means the non-overlined parts are not divisible by $\ell$). Alanazi, Alenazi, Keith, and Munagi explored several identities involving $\overline{R_\ell^\ast}(n)$ and other restricted partition functions and established various congruences modulo $3$ using elementary dissection techniques. For instance, they proved that for all $n \geq 0$ and $j \geq 3$, we have  
\[
\overline{R_3^\ast}(9n+4) \equiv  
\overline{R_3^\ast}(9n+7) \equiv 
\overline{R_{3^j}^\ast}(27n+19) \equiv 0 \pmod{3}.  
\]
The generating function for $\overline{R_\ell^\ast}(n)$ is given by \cite{AlanaziAlenaziKeithMunagi}  
\begin{equation}\label{eq:gf-rast}  
   \sum_{n\geq 0}\overline{R_\ell^\ast}(n)q^n = \frac{f_2 f_\ell}{f_1^2}.  
\end{equation} 

Subsequently, Sellers \cite{sell} established that, for all $n\geq 0$,
\[ 
\overline{R_3^\ast}(9n+4) \equiv 
\overline{R_3^\ast}(9n+7) \equiv 0 \pmod{4}.
\]
Additionally, he also proved two infinite families of congruences using induction: namely, for all $k\geq 1$ and all $n\geq 0$, we have
\begin{align*}
    \overline{R_3^\ast}\left( 9^k n + \frac{33 \cdot 9^{k-1}-1}{8} \right) & \equiv 0 \pmod{3},\\
    \overline{R_3^\ast}\left( 9^k n + \frac{57 \cdot 9^{k-1}-1}{8} \right) & \equiv 0 \pmod{3}.
\end{align*}
More recently, Alanazi, Munagi, and Saikia \cite{ala} established the following congruences for all $n \geq 0$:  
\begin{align}
   \overline{ R_3^\ast}(9n + 4) &\equiv 0 \pmod{12},\label{mun-1} \\
  \overline{   R_3^\ast}(9n + 7) &\equiv 0 \pmod{48}, \label{mun-2}\\
   \overline{  R_6^\ast}(9n + 5) &\equiv 0 \pmod{24}, \label{mun-3}\\
   \overline{  R_6^\ast}(9n + 8) &\equiv 0 \pmod{96}\label{mun-4}.
\end{align}
They mentioned that these congruences can be proved using the \texttt{RaduRK} Mathematica package. Additionally, they discovered the following two congruences but only mentioned that an automatic proof is available:
\begin{align}  
\overline{R_6^\ast}(27n+11) &\equiv 0 \pmod{64}, \label{mun-5}\\  
\overline{R_6^\ast}(81n+47) &\equiv 0 \pmod{24}.\label{mun-6}
\end{align}

In this paper, we are interested in finding new arithmetic properties, including infinite families of congruences for the $\overline{R^\ast_\ell}(n)$ function for different values of $\ell$. Since our proof techniques involve both elementary means as well as the theory of modular forms, we first state the results that we prove using elementary means, and then go to the results which we prove via the theory of modular forms.

From the work of Alanazi et al. \cite{AlanaziAlenaziKeithMunagi} and Sellers \cite{sell} we have an elementary proof of \eqref{mun-1}. Alanazi et al. \cite{ala} asked for elementary proofs of their congruences, and we give such proofs for \eqref{mun-2} -- \eqref{mun-4} in this paper. We also include \eqref{mun-1} as our method naturally gives a proof of that as well.

\begin{theorem}\label{elthm}
    For all $n\geq 0$,
    \begin{align}
   \overline{   R_3^\ast}(9n + 4) &\equiv 0 \pmod{12}, \label{amun-1}\\ 
   \overline{   R_3^\ast}(9n + 7) &\equiv 0 \pmod{48}, \label{amun-2}\\ 
   \overline{  R_6^\ast}(9n + 5) &\equiv 0 \pmod{24}, \label{amun-3}\\
   \overline{  R_6^\ast}(9n + 8) &\equiv 0 \pmod{96}\label{amun-4}.
    \end{align}
\end{theorem}
\noindent Theorem \ref{elthm} is proved in Section \ref{sec-elmthm}. The same proof technique also gives us the following family of congruences.
    \begin{theorem}\label{thm3n}
For all $n \geq 0$ and $k\geq 1$, we have
\begin{align*}
\overline{R_{3k}^\ast}(3n+1) &\equiv 0 \pmod{2}, \\
\overline{R_{3k}^\ast}(3n+2) &\equiv 0 \pmod{4}.
\end{align*}
\end{theorem}
\noindent Theorem \ref{thm3n} is proved in Section \ref{sec-elmthm}.

We mention in passing the following general congruence which has not been observed before.
\begin{theorem}\label{thm1.000}
For all $n\geq 0$ and $\ell\geq 1$, we have
    \begin{equation*}
    \overline{R_{2\ell}^\ast}(2n+1)\equiv 0 \pmod{2}.
    \end{equation*}
\end{theorem}
\noindent Theorem \ref{thm1.000} follows very easily by applying \eqref{1/f1^2} in \eqref{eq:gf-rast}. We skip the proof of this result. We also have the following straightforward observation.
\begin{theorem}
For all $n \geq 0$ and $\ell \nmid r$, we have
   $$\overline{R_\ell^\ast}(\ell n + r) \equiv 0 \pmod{2}.$$ 
\end{theorem}

\begin{proof}
Modulo \(2\), we have
\begin{equation*}
    \sum_{n \geq 0} \overline{R_\ell^\ast}(n)q^n \equiv f_{\ell} \pmod{2}.
\end{equation*}
By extracting the terms of the form \(q^{\ell n + r}\), where \(\ell \nmid n\), the result follows.
\end{proof}

Furthermore, the following simple observation naturally follows from Euler's pentagonal number theorem \cite[Eq. (1.3.18)]{Spirit}:
\begin{equation}\label{e2.0.3.4}
    f_1=\sum_{n \in \mathbb{Z}}(-1)^nq^{n(3n-1)/2}.
\end{equation}

\begin{theorem}\label{thm1}
   Let $\ell \geq 2$.  Then, for all $n\geq 0$, $\overline{R_{\ell}^\ast}(n)$ is odd if and only if $\frac{24n}{\ell}+1$ is a square.
\end{theorem}

\begin{proof}
Modulo \(2\), we have
    \begin{equation*}
    \sum_{n \geq 0} \overline{R_\ell^\ast}(n)q^n \equiv f_{\ell} \pmod{2}.
\end{equation*}
This means that $\overline{R_\ell^\ast}(n)$ is odd if and only if $n = \ell\cdot k(3k-1)/2$ for some $k$ thanks to \eqref{e2.0.3.4}. Then $24n+\ell = \ell(6k-1)^2$ or $\frac{24n}{\ell}+1 = (6k-1)^2$.  The result follows.
\end{proof}

We next prove several Ramanujan--like congruences modulo small powers of 2 satisfied by $\overline{R^\ast_8}$.

\begin{theorem}\label{thm:nathsel}
For all $n\geq 0$, we have
\begin{align}
\overline{R^\ast_8}(16n+9)&\equiv 0 \pmod{8},\label{16n9}\\
\overline{R^\ast_8}(16n+11)&\equiv 0 \pmod{16},\label{16n11}\\
\overline{R^\ast_8}(32n+25)&\equiv 0 \pmod{16},\label{32n25}\\
\overline{R^\ast_8}(64n+37)&\equiv 0 \pmod{32},\label{64n37}\\
\overline{R^\ast_8}(16n+13)&\equiv 0 \pmod{64},\label{16n13}\\
\overline{R^\ast_8}(32n+21)&\equiv 0 \pmod{64},\label{32n21}\\
\overline{R^\ast_8}(16n+15)&\equiv 0 \pmod{128},\label{16n15}\\
\overline{R^\ast_8}(32n+29)&\equiv 0 \pmod{256},\label{32n29}\\
\overline{R^\ast_8}(128n+85)&\equiv 0 \pmod{512},\label{128n85}\\
\overline{R^\ast_8}(64n+53)&\equiv 0 \pmod{1024},\label{64n53}\\
\overline{R^\ast_8}(128n+117)&\equiv 0 \pmod{4096}.\label{128n117}
\end{align}
\end{theorem}

\noindent Theorem \ref{thm:nathsel} is proved in Section \ref{sec-elmthm}.

We next prove the following infinite family of congruences.
\begin{theorem}\label{saik}
     For all $n\geq 0$ and $k\geq 3$, we have
    \begin{align*}
        \overline{R^\ast_{2^k}}(8n+1)&\equiv 0\pmod2,\\
    \overline{R^\ast_{2^k}}(8n+2)&\equiv 0\pmod2,\\
    \overline{R^\ast_{2^k}}(8n+3)&\equiv 0\pmod4,\\
    \overline{R^\ast_{2^k}}(8n+4)&\equiv 0\pmod2,\\
    \overline{R^\ast_{2^k}}(8n+5)&\equiv 0\pmod{8},\\
    \overline{R^\ast_{2^k}}(8n+6)&\equiv 0\pmod4,\\
    \overline{R^\ast_{2^k}}(8n+7)&\equiv 0\pmod{16}.
    \end{align*}
\end{theorem}

\noindent Theorem \ref{saik} is proved in Section \ref{sec-elm2}. It can be noted that the $8n+2i$ cases, $1\leq i\leq 3$,  follow from Theorem \ref{thm1.000}. However, we can give a uniform proof of all the above congruences, so we include these cases as well in the statement.

We now state three results which follow from applications of classical results involving generating function manipulations.

\begin{theorem}\label{thmjames1}
    For all $n\geq 0$, we have
    \[\overline{R_6^\ast}(6n+5) = \overline{R_6^\ast}(3(2n+1)+2) \equiv 0 \pmod{8}.\]
\end{theorem}

\begin{theorem}\label{thmjames2}
    For all $n\geq 0$, we have \[\overline{R_6^\ast}(3(3n+1)+2) \equiv \overline{R_6^\ast}(3(3n+2)+2) \equiv0 \pmod{8}.\]  That is, for all $n\geq 0$, we have \[\overline{R_6^\ast}(9n+5) \equiv \overline{R_6^\ast}(9n+8) \equiv0 \pmod{8}.\] 
\end{theorem}

\begin{theorem}\label{thmjames3}
    Let $p\geq 5$ be prime, and let $r$, $1\leq r\leq p-1$, be such that $\text{inv}(3,p)\cdot 4r +1$ is a quadratic nonresidue modulo $p$ where $inv(3,p)$ is the inverse of 3 modulo $p$.  Then, for all $n\geq 0$, we have \[\overline{R_6^\ast}(3(pn+r)+2) \equiv 0 \pmod{8}.\]
\end{theorem}

\noindent Theorems \ref{thmjames1}--\ref{thmjames3} are proved in Section \ref{sec-elm3}.

All of the above results are proved using elementary techniques. We now move on to results which we prove using the theory of modular forms.

\begin{theorem}\label{thm1.00}
	Let $k, n$ be nonnegative integers. For each $i$ with $1\leq i \leq k+1$, if $p_i \geq 3$ is prime such that $p_i \not\equiv 1 \pmod 8$, then for any integer $j \not\equiv 0 \pmod {p_{k+1}}$
	\begin{align*} 
	\overline{R_6^\ast}\left(18p_1^2\dots p_{k+1}^2n + \frac{9p_1^2\dots p_{k}^2p_{k+1}(4j+p_{k+1})-1}{4}\right) \equiv 0 \pmod{8}.
\end{align*}
\end{theorem}
\noindent Theorem \ref{thm1.00} is proved in Section \ref{sec:mf1}. There are arithmetic progressions covered by Theorem \ref{thm1.00}, which are not covered by Theorems \ref{thmjames1} and \ref{thmjames2} above. For instance, if we consider $k=1$, $p_1=11$, $p_2=13$ and $j=2$ in Theorem \ref{thm1.00}, then for all $n\geq 0$ we get 
\begin{equation*}
    \overline{R_6^\ast}\left(368082n+74324\right) \equiv 0 \pmod{8}.
\end{equation*}
Let $p\geq 3$ be a prime such that $p \not\equiv 1\pmod{8}$. By taking all the primes $p_1, p_2, \ldots, p_{k+1}$ to be equal to the same prime $p$ in Theorem \ref{thm1.00}, 
we obtain the following infinite family of congruences for $\overline{R_6^\ast}(n)$:
	\begin{align*} 
	\overline{R_6^\ast}\left( 18p^{2(k+1)}n + 9p^{2k+1}j + \frac{9 p^{2(k+1)}-1}{4}\right) \equiv 0 \pmod 8,
	\end{align*}
	where  $j \not\equiv 0 \pmod p$. As an example, for all $n\geq 0$ and $j\not\equiv 0\pmod{3}$, taking $k=0$ and $p=3$, we have
	\begin{align*} 
	\overline{R_6^\ast}\left(162n + 27j + 20\right) \equiv 0 \pmod{8}.
	\end{align*}
\noindent We also note that the above congruence is not covered by Theorem \ref{thmjames2}.

For the next result, we require the concept of the Legendre symbol, which for a prime $p\geq3$ is defined as
\begin{align*}
\left(\dfrac{s}{p}\right)_L:=\begin{cases}\quad1,\quad \text{if $s$ is a quadratic residue modulo $p$ and $p\nmid s$,}\\\quad 0,\quad \text{if $p\mid s$,}\\~-1,\quad \text{if $s$ is a quadratic nonresidue modulo $p$.}
\end{cases}
\end{align*}

We need some notation before we state the next result. The Dedekind $\eta$- function is given by
\begin{equation*}
    \eta(z) := q^{1/24}f_1,
\end{equation*}
where $q=e^{2\pi iz}$ and $z$ lies in the complex upper half plane $\mathbb{H}$. \\

The well known $\Delta$-function is denoted by
\begin{equation*}
    \Delta(z) := \eta(z)^{24} = \sum_{n \geq 1}\tau(n)q^n.
\end{equation*}

\begin{theorem}\label{t4}
Let $p$ be an odd prime. Suppose that $s$ is an integer satisfying $1 \leq s \leq 8p$, $s \equiv 1 \pmod{8}$, and $\left( \dfrac{s}{p} \right)_L = -1$. Then, we have  
\begin{equation}\label{e12.0.1}
    \overline{R_6^\ast}\left( 18pn + \dfrac{9s-1}{4} \right) \equiv 0 \pmod{8}.
\end{equation}

Furthermore, if $\tau(p) \equiv 0 \pmod{2}$, then for all $n \geq 0$ and $k \geq 1$, we have  
\begin{equation}\label{e12.0.2}
    \overline{R_6^\ast}\left( 18p^{2k+1}n+\dfrac{9sp^{2k}-1}{4} \right) \equiv 0 \pmod{8}.
\end{equation}
\end{theorem}

\begin{example}
Substituting $s = 17$ and $p = 3$ into equations \eqref{e12.0.1} and \eqref{e12.0.2}, respectively, we obtain the following congruences, valid for all $n \geq 0$:
\begin{align*}
    \overline{R_6^\ast}\left(54n + 38\right) &\equiv 0 \pmod{8},\\
    \overline{R_6^\ast}\left(486n + 344\right) &\equiv 0 \pmod{8}.
\end{align*}
\end{example}

\noindent \textcolor{blue}{Theorem \ref{t4} is} proved in Section \ref{sec:mf2}.

This paper is organized as follows: in Section \ref{sec:prelim} we mention some preliminaries which are required for our proofs, Sections \ref{sec-elmthm} -- \ref{sec:mf2} contains the proofs of our results, and the paper ends with some concluding remarks in Section \ref{sec:concl}.

\section{Preliminaries}\label{sec:prelim}

\subsection{Elementary Results}

We need Jacobi's Triple Product identity \cite[Eq. (1.3.24)]{Spirit}
\begin{equation}\label{e2.0.3.3}
    f_1^3=\sum_{n\geq 0}(-1)^n(2n+1)q^{n(n+1)/2}.
\end{equation}

We also need the following $2$-dissection formulas.
\begin{lemma}
    We have 
    \begin{align}
        \frac{1}{f_{1}^2} &= \frac{f_{8}^5}{f_{2}^5 f_{16}^2} + 2 q \frac{f_{4}^2 f_{16}^2 }{f_{2}^5 f_{8}},\label{1/f1^2}\\
        \frac{1}{f_{1}^4} &= \frac{f_{4}^{14}}{f_{2}^{14} f_{8}^4} + 4 q \frac{f_{4}^2 f_{8}^4}{f_{2}^{10}}.\label{1/f1^4}
    \end{align}
\end{lemma}
\begin{proof}
    Equation \eqref{1/f1^2} is \cite[Eq. (1.9.4)]{power} and Equation \eqref{1/f1^4} can be found in \cite[Lemma 2.2]{SellersJMAA}.   
\end{proof}

\begin{lemma}
    We have the following generating functions
    \begin{align}
\sum_{n\geq 0}\overline{R_8^\ast}(2n)q^n&=\frac{f_4^6}{f_1^4f_8^2}, \label{e 2n}\\
\sum_{n\geq 0}\overline{R_8^\ast}(2n+1)q^n&=2\frac{f_2^2f_8^2}{f_1^4},\label{e 2n+1}.
\end{align}
\end{lemma}

\begin{proof}
    Equations \eqref{e 2n} and \eqref{e 2n+1} follow from applying \eqref{1/f1^2} in \eqref{eq:gf-rast} with $\ell = 8$.
\end{proof}

We need the following $3$-dissection formula.
\begin{lemma}
    We have 
\begin{equation}\label{HirSel3diss} \frac{f_2}{f_1^2}=\frac{f_6^4f_9^6}{f_3^8f_{18}^3}+2q\frac{f_6^3f_9^3}{f_3^7}+4q^2\frac{f_6^2f_{18}^3}{f_3^6}.
\end{equation}
\end{lemma}
\begin{proof}
For a proof of the above, see \cite[Theorem 6]{hirschhorn2006arithmetic}.    
\end{proof}

Finally, we have the following lemma.
\begin{lemma}\label{lm}
    We have
    \begin{equation}\label{e0.1}
       \sum_{n\geq 0}\overline{R_6^\ast}(18n+2)q^n \equiv 4f_1^3 \pmod{8}. 
    \end{equation}
\end{lemma}

\begin{proof}
 Using \eqref{eq:gf-rast} and \eqref{HirSel3diss}, we obtain
    \begin{equation}\label{e1.010101}
        \sum_{n\geq 0}\overline{R_6^\ast}(3n+2)q^n = 4\frac{f_2^3f_6^3}{f_1^6} \equiv 4\frac{f_1^6f_6^3}{f_1^6} \equiv 4f_6^3 \pmod{8}. 
    \end{equation}

Collecting the terms of the form $q^{6n}$ from both sides of the equation \eqref{e1.010101}, we get
		\begin{align}\label{z1}
		\sum_{n \geq 0}\overline{R_6^\ast}(18n+2)q^n \equiv 4f_1^3 \pmod{8}.
		\end{align}
\end{proof}

\subsection{Background related to modular forms}

We recall some basic facts and results from the theory of modular forms that will be useful in the proof of some of our results.

For a positive integer $N$, we will assume that:
\begin{align*}
\textup{SL}_2(\mathbb{Z}) & :=\left\{\begin{bmatrix}
a  &  b \\
c  &  d      
\end{bmatrix}: a, b, c, d \in \mathbb{Z}, ad-bc=1
\right\},\\
\Gamma_{0}(N) &:=\left\{
\begin{bmatrix}
a  &  b \\
c  &  d      
\end{bmatrix} \in \Gamma : c\equiv~0\pmod N \right\},\\
\Gamma_{1}(N) &:=\left\{
\begin{bmatrix}
a  &  b \\
c  &  d      
\end{bmatrix} \in \Gamma : a\equiv d \equiv 1\pmod N \right\},\\
\Gamma(N) &:=\left\{
\begin{bmatrix}
a  &  b \\
c  &  d      
\end{bmatrix} \in \textup{SL}_2(\mathbb{Z}) : a \equiv d \equiv 1 \pmod{N}, \text{and} \hspace{2mm} b\equiv c \equiv 0 \pmod N \right\}.
\end{align*}
A subgroup $\Gamma$ of $\textup{SL}_2(\mathbb{Z})$ is called a congruence subgroup if $\Gamma(N) \subseteq \Gamma$ for some $N$. The smallest $N$ such that $\Gamma(N) \subseteq \Gamma$ is called the level of $\Gamma$. For example, $\Gamma_{0}(N)$ and $\Gamma_{1}(N)$ are congruence subgroups of level $N$.\\

Let $\mathbb{H}:= \{ z \in \mathbb{C} : \Im(z) > 0 \}$ be the upper half of the complex plane. The group 

$$\textup{GL}_2^{+}(\mathbb{R}) = \left\{
\begin{bmatrix}
a  &  b \\
c  &  d      
\end{bmatrix} : a,b,c,d \in \mathbb{R} \hspace{2mm} \text{and} \hspace{2mm} ad-bc>0 \right\}$$ acts on $\mathbb{H}$ by $\begin{bmatrix}
a  &  b \\
c  &  d      
\end{bmatrix} z = \dfrac{az +b}{cz+d}$. We identify $\infty$ with $\dfrac{1}{0}$ and define $ \begin{bmatrix}
a  &  b \\
c  &  d      
\end{bmatrix} \dfrac{r}{s} = \dfrac{ar +bs}{cr+ds}$, where $\dfrac{r}{s} \in \mathbb{Q} \cup \{\infty\}$. This gives an action of $\textup{GL}_2^{+}(\mathbb{R})$ on the extended upper half-plane $\mathbb{H}^{\star} = \mathbb{H} \cup \mathbb{Q} \cup \{\infty\}$. Suppose that $\Gamma$ is a congruence subgroup of $\textup{SL}_2(\mathbb{Z})$. A cusp of $\Gamma$ is an equivalence class in $\mathbb{P}^{1}=\mathbb{Q} \cup \{ \infty\}$ under the action of $\Gamma$.\\

The group $\textup{GL}_2^{+}(\mathbb{R})$ also acts on the functions $f : \mathbb{H} \to \mathbb{C}$. In particular, suppose that $\gamma = \begin{bmatrix}
a  &  b \\
c  &  d      
\end{bmatrix} \in \textup{GL}_2^{+}(\mathbb{R}) $. If $f(z)$ is a meromorphic function on $\mathbb{H}$ and $\ell$ is an integer, then define the slash operator $\mid_{\ell}$ by 
$(f\mid_{\ell}\gamma )(z) := (det \gamma)^{\ell/2}(cz+d)^{-\ell}f(\gamma z)$.\\

\begin{definition}
Let $\gamma$ be a congruence subgroup of level N. A holomorphic function $f : \mathbb{H} \to \mathbb{C}$ is called a modular form with integer weight on $\Gamma$ if the following hold:
\begin{enumerate}
    \item We have
    \begin{align*}
        f\left(\frac{az+b}{cz+d}\right)=(cz+d)^{\ell}f(z)
    \end{align*}
    for all $z\in\mathbb{H}$ and all $\begin{bmatrix}
        a&b\\
        c&d
    \end{bmatrix}\in\Gamma$.
    \item If $\gamma\in\textup{SL}_2(\mathbb{Z})$, then $(f\mid_{\ell}\gamma )(z)$ has a Fourier expansion of the form
    \begin{align*}
        (f\mid_{\ell}\gamma )(z)=\sum_{n \geq 0}a_{\gamma}(n)q^n_{N},
    \end{align*}
    where $q^n_{N}=e^{\frac{2\pi iz}{N}}$.
\end{enumerate}
\end{definition}

For a positive integer $\ell$, let $M_\ell(\Gamma_1(N)))$ denote the complex vector space of modular forms of weight $\ell$ with respect to $\Gamma_1(N)$. 
\begin{definition}\cite[Definition 1.15]{ono2004}
If $\chi$ is a Dirichlet character modulo $N$, then a modular form $f\in M_\ell(\Gamma_1(N))$ has Nebentypus character $\chi$ if 
$f\left( \frac{az+b}{cz+d}\right)=\chi(d)(cz+d)^{\ell}f(z)$ for all $z\in \mathbb{H}$ and all $\begin{bmatrix}
a  &  b \\
c  &  d      
\end{bmatrix}\in \Gamma_0(N)$. The space of such modular forms is denoted by $M_\ell(\Gamma_0(N),\chi)$.
\end{definition}

\begin{theorem}\cite[Theorem 1.64]{ono2004}\label{thm_ono1} Suppose that $f(z)=\displaystyle\prod_{\delta\mid N}\eta(\delta z)^{r_\delta}$ 
		is an eta-quotient such that $\ell=\displaystyle\dfrac{1}{2}\sum_{\delta\mid N}r_{\delta}\in \mathbb{Z}$, $\sum_{\delta\mid N} \delta r_{\delta}\equiv 0 \pmod{24}$, and  $\sum_{\delta\mid N} \dfrac{N}{\delta}r_{\delta}\equiv 0 \pmod{24}$.
		Then, 
		$
		f\left( \dfrac{az+b}{cz+d}\right)=\chi(d)(cz+d)^{\ell}f(z)
		$
		for every  $\begin{bmatrix}
			a  &  b \\
			c  &  d      
		\end{bmatrix} \in \Gamma_0(N)$. Here 
\(
		    \chi(d):=\left(\dfrac{(-1)^{\ell} \prod_{\delta\mid N}\delta^{r_{\delta}}}{d}\right)_L.\label{chi}\)
	\end{theorem}

\noindent If the eta-quotient $f(z)$ satisfies the conditions of Theorem \ref{thm_ono1} and  is holomorphic at all of the cusps of $\Gamma_0(N)$, then $f\in M_{\ell}(\Gamma_0(N), \chi)$.

\begin{definition}
Let $m$ be a positive integer and $f(z) = \sum_{n \geq 0} a(n)q^n \in M_{k}(\Gamma_0(N),\chi)$, where $\chi$ is a Dirichlet character modulo $N$. Then the action of Hecke operator $T_m$ on $f(z)$ is defined by 
\begin{align*}
f(z)|T_m := \sum_{n \geq 0} \left(\sum_{d\mid \gcd(n,m)}\chi(d)d^{k-1}a\left(\frac{nm}{d^2}\right)\right)q^n.
\end{align*}
In particular, if $m=p$ is prime, we have 
\begin{align}\label{hecke1}
f(z)|T_p := \sum_{n \geq 0} \left(a(pn)+\chi(p)p^{k-1}a\left(\frac{n}{p}\right)\right)q^n.
\end{align}
We note that $a(n)=0$ unless $n$ is a nonnegative integer.
\end{definition}
\begin{definition}\label{hecke2}
	A modular form $f(z)=\sum_{n \geq 0}a(n)q^n \in M_{k}(\Gamma_0(N),\chi)$ is called a Hecke eigenform if for every $m\geq2$ there exists a complex number $\lambda(m)$ for which 
	\begin{align}\label{hecke3}
	f(z)|T_m = \lambda(m)f(z).
	\end{align}
\end{definition}

In the special case when $\chi$ is the trivial character (i.e., $\chi(n) = 1$ for all $(n,N) = 1$), the Hecke operator simplifies to
\begin{align*}
f(z)|T_m := \sum_{n \geq 0} \left(\sum_{d\mid \gcd(n,m)}d^{k-1}a\left(\frac{nm}{d^2}\right)\right)q^n.
\end{align*}

\section{Proofs of Theorems \ref{elthm}, \ref{thm3n}, and \ref{thm:nathsel}}\label{sec-elmthm}

\begin{proof}[Proof of Theorem \ref{elthm}]
We have
\[
\sum_{n\geq 0}\overline{R^\ast_3}(n)q^n=\frac{f_2f_3}{f_1^2}.
\]
Using \eqref{HirSel3diss} in the above and extracting the terms involving $q^{3n+1}$, we find
\begin{equation}\label{mun3n1}
    \sum_{n\geq 0}\overline{R^\ast_3}(3n+1)q^n=2\frac{f_2^3f_3^3}{f_1^6}.
\end{equation}
Applying \eqref{HirSel3diss} on \eqref{mun3n1}, we obtain
\begin{equation}\label{nn1}
    \sum_{n\geq 0}\overline{R^\ast_3}(3n+1)q^n=2f_3^3\left(\frac{f_6^4f_9^6}{f_3^8f_{18}^3}+2q\frac{f_6^3f_9^3}{f_3^7}+4q^2\frac{f_6^2f_{18}^3}{f_3^6}\right)^3.
\end{equation}
Now, extracting the terms involving $q^{3n+1}$ from \eqref{nn1}, we have
\begin{equation*}
  \sum_{n\geq 0}\overline{R^\ast_3}(9n+4)q^n= 12\left( \dfrac{f_2^{11}f_3^{15}}{f_1^{20}f_6^6}+16q\dfrac{f_2^8f_3^6f_6^3}{f_1^{17}} \right).
\end{equation*}
This proves \eqref{amun-1}. Again, extracting the terms involving $q^{3n+2}$ from \eqref{nn1}, we have
\begin{align*}
  \sum_{n\geq 0}\overline{R^\ast_3}(9n+7)q^n&=2f_1^3\left(24\frac{f_2^{10}f_3^{12}}{f_1^{22}f_6^3}+96q\frac{f_2^7f_3^3f_6^6}{f_1^{19}}\right)\\
  &= 48\left(\frac{f_2^{10}f_3^{12}}{f_6^3f_1^{19}}+4q\frac{f_2^7f_3^3f_6^6}{f_1^{16}}\right).
\end{align*}
This proves \eqref{amun-2}.

Shifting to the case $\ell=6$, we have
\[
\sum_{n\geq 0}\overline{R^\ast_6}(n)q^n=\frac{f_6f_2}{f_1^2}.
\]
Using \eqref{HirSel3diss} in the above and extracting the terms involving $q^{3n+2}$ we obtain
\begin{equation}\label{mun6n}
   \sum_{n\geq 0}\overline{R^\ast_6}(3n+2)q^n =4\frac{f_2^3f_6^3}{f_1^6}.
\end{equation}
Comparing \eqref{mun6n} with \eqref{mun3n1} we see that
\[
\sum_{n\geq 0}\overline{R^\ast_6}(3n+2)q^n =2\frac{f_6^3}{f_3^3}\sum_{n\geq 0}\overline{R^\ast_3}(3n+1)q^n.
\]
This means that applying a process similar to the process we applied to \eqref{mun3n1}, we can prove \eqref{amun-4} (since $f_6^3/f_3^3$ is a function of $q^3$ ). For the sake of brevity, we leave the relevant details to the interested reader.

On the other hand, from \eqref{mun6n} and \eqref{HirSel3diss} we have
\begin{equation}\label{nn2}
\sum_{n\geq 0}\overline{R^\ast_6}(3n+2)q^n =4f_6^3\left(\frac{f_6^4f_9^6}{f_3^8f_{18}^3}+2q\frac{f_6^3f_9^3}{f_3^7}+4q^2\frac{f_6^2f_{18}^3}{f_3^6}\right)^3.
\end{equation}
Extracting the terms involving $q^{3n+1}$ from \eqref{nn2}, we have
\begin{align*}
    \sum_{n\geq 0}\overline{R^\ast_6}(9n+5)q^n &=4f_2^3\left(6\frac{f_2^{11}f_3^{15}}{f_1^{23}f_6^6}+96q\frac{f_2^8f_3^6f_6^3}{f_1^{20}}\right)\\
    &=24 \left(\frac{f_2^{14}f_3^{15}}{f_1^{23}f_6^6}+16q\frac{f_2^{11}f_3^6f_6^3}{f_1^{20}}\right).
\end{align*}
This proves \eqref{amun-3}. Again, extracting the terms involving $q^{3n+2}$ from \eqref{nn2}, we have
\begin{align*}
    \sum_{n\geq 0}\overline{R^\ast_6}(9n+8)q^n &= 96\left( \dfrac{f_2^{13}f_3^{12}}{f_1^{22}f_6^3}+4q\dfrac{f_2^{10}f_3^3f_6^6}{f_1^{19}} \right).
\end{align*}
This proves \eqref{amun-4}.
\end{proof}

\begin{proof}[Proof of Theorem \ref{thm3n}]
The proof is immediate if we apply \eqref{HirSel3diss} to the following
\[
\sum_{n\geq 0}\overline{R^\ast_{3k}}(n)q^n=\frac{f_2f_{3k}}{f_1^2},
\]
and then extract the terms involving $q^{3n+1}$ and $q^{3n+2}$.
\end{proof}

\begin{proof}[Proof of Theorem \ref{thm:nathsel}]
    Since the proofs are all very similar, we only give a complete proof of five congruences, and mention the details for the other congruences.

Employing \eqref{1/f1^4} to \eqref{e 2n+1}, and then extracting the terms of the form $q^{2n}$, we obtain
\begin{equation}\label{e 4n+1}
    \sum_{n\geq 0}\overline{R_8^\ast}(4n+1)q^n=2\frac{f_2^{14}}{f_1^{12}f_4^2}.
\end{equation}
Employing \eqref{1/f1^4} to \eqref{e 2n+1}, and then extracting the terms of the form $q^{2n+1}$, we obtain
\begin{equation}\label{e 4n+3}
    \sum_{n\geq 0}\overline{R_8^\ast}(4n+3)q^n=8\frac{f_2^2f_4^6}{f_1^8}.
\end{equation}
Employing \eqref{1/f1^4} to \eqref{e 4n+1}, and then extracting the terms of the form $q^{2n}$, we obtain
\begin{equation}\label{e 8n+1}
    \sum_{n\geq 0}\overline{R_8^\ast}(8n+1)q^n=2\left(\frac{f_2^{40}}{f_1^{28}f_4^{12}}+48q\frac{f_4^4f_2^{16}}{f_1^{20}}\right).
\end{equation}
Employing \eqref{1/f1^4} to \eqref{e 4n+3}, and then extracting the terms of the form $q^{2n}$, we obtain
\begin{equation}\label{e 8n+3}
    \sum_{n\geq 0}\overline{R_8^\ast}(8n+3)q^n=8\left(\frac{f_2^{34}}{f_1^{26}f_4^8}+16q\frac{f_4^8f_2^{10}}{f_1^{18}}\right).
\end{equation}
Employing \eqref{1/f1^4} to \eqref{e 4n+1}, and then extracting the terms of the form $q^{2n+1}$, we obtain
\begin{equation}\label{e 8n+5}
    \sum_{n\geq 0}\overline{R_8^\ast}(8n+5)q^n=8\left(3\frac{f_2^{28}}{f_1^{24}f_4^4}+16q\frac{f_4^{12}f_2^4}{f_1^{16}}\right).
\end{equation}
Employing \eqref{1/f1^4} \textcolor{blue}{to \eqref{e 4n+3}}, and then extracting the terms of the form $q^{2n+1}$, we obtain
\begin{equation}\label{e 8n+7}
    \sum_{n\geq 0}\overline{R_8^\ast}(8n+7)q^n=64\frac{f_2^{22}}{f_1^{22}}.
\end{equation}
Employing \eqref{1/f1^4} to \eqref{e 8n+5}, and then extracting the terms of the form $q^{2n}$, we obtain
\begin{equation}\label{e 16n+5}
    \sum_{n\geq 0}\overline{R_8^\ast}(16n+5)q^n=8\left(3\frac{f_2^{80}}{f_1^{56}f_4^{24}}+976q\frac{f_2^{56}}{f_1^{48}f_4^8}+15616q^2\frac{f_4^8f_2^{32}}{f_1^{40}}+12288q^3\frac{f_4^{24}f_2^8}{f_1^{32}}\right).
\end{equation}
Employing \eqref{1/f1^4} to \eqref{e 8n+1}, and then extracting the terms of the form $q^{2n+1}$, we obtain
\begin{equation}\label{e 16n+9}
    \sum_{n\geq 0}\overline{R_8^\ast}(16n+9)q^n=8\left(19\frac{f_2^{74}}{f_1^{54}f_4^{20}}+2480q\frac{f_2^{50}}{f_1^{46}f_4^4}+20736q^2\frac{f_4^{12}f_2^{26}}{f_1^{38}}+4096q^3\frac{f_4^{28}f_2^2}{f_1^{30}}\right).
\end{equation}
This proves \eqref{16n9}.

Employing \eqref{1/f1^2} to \eqref{e 8n+3}, and then extracting the terms of the form $q^{2n+1}$, we obtain
\begin{align}\label{e 16n+11}
    \sum_{n\geq 0}\overline{R_8^\ast}(16n+11)q^n&=16\left(8\frac{f_2^8f_4^{45}}{f_1^{35}f_8^{18}}+13\frac{f_4^{59}}{f_1^{31}f_2^6f_8^{22}}+1144q\frac{f_4^{47}}{f_1^{31}f_2^2f_8^{14}}+1152q\frac{f_2^{12}f_4^{33}}{f_1^{35}f_8^{10}} \right. \nonumber \\ 
&\quad +16128q^2\frac{f_2^{16}f_4^{21}}{f_1^{35}f_8^2}+20592q^2\frac{f_2^2f_4^{35}}{f_1^{31}f_8^6}+43008q^3\frac{f_2^{20}f_8^6f_4^9}{f_1^{35}} \nonumber  \\
&\quad +109824q^3\frac{f_2^6f_8^2f_4^{23}}{f_1^{31}}+18432q^4\frac{f_2^{24}f_8^{14}}{f_1^{35}f_4^3}+ 183040q^4\frac{f_2^{10}f_8^{10}f_4^{11}}{f_1^{31}} \nonumber  \\
&\quad \left. +79872q^5\frac{f_2^{14}f_8^{18}}{f_1^{31}f_4}+4096q^6\frac{f_2^{18}f_8^{26}}{f_1^{31}f_4^{13}}\right).
\end{align}
This proves \eqref{16n11}.

Employing \eqref{1/f1^4} to \eqref{e 8n+5}, and then extracting the terms of the form $q^{2n+1}$, we obtain
\begin{equation}\label{e 16n+13}
    \sum_{n\geq 0}\overline{R_8^\ast}(16n+13)q^n=64\left(11\frac{f_2^{68}}{f_1^{52}f_4^{16}}+672q\frac{f_2^{44}}{f_1^{44}}+2816q^2\frac{f_4^{16}f_2^{20}}{f_1^{36}}\right).
\end{equation}
This proves \eqref{16n13}.

Employing \eqref{1/f1^2} to \eqref{e 8n+7}, and then extracting the terms of the form $q^{2n+1}$, we obtain
\begin{align}\label{e 16n+15}
\sum_{n\geq 0}\overline{R_8^\ast}(16n+15)q^n&=128\left(11\frac{f_2^2f_4^{49}}{f_1^{33}f_8^{18}}+660q\frac{f_2^6f_4^{37}}{f_1^{33}f_8^{10}}+7392q^2\frac{f_2^{10}f_4^{25}}{f_1^{33}f_8^2}+21120q^3\frac{f_2^{14}f_8^6f_4^{13}}{f_1^{33}}  \right. \nonumber  \\
&\quad \left.+14080q^4\frac{f_2^{18}f_8^{14}f_4}{f_1^{33}}+1024q^5\frac{f_2^{22}f_8^{22}}{f_1^{33}f_4^{11}}  \right).
\end{align}
This proves \eqref{16n15}.

Employing \eqref{1/f1^4} to \eqref{e 16n+5}, and then extracting the terms of the form $q^{2n+1}$, we obtain
\begin{align}\label{e 32n+21}
    \sum_{n\geq 0}\overline{R_8^\ast}(32n+21)q^n&=64\left(143\frac{f_2^{160}}{f_1^{112}f_4^{48}}+217184q\frac{f_2^{136}}{f_1^{104}f_4^{32}}+31908096q^2\frac{f_2^{112}}{f_1^{96}f_4^{16}} \right. \nonumber \\ 
&\quad +1014054912q^3\frac{f_2^{88}}{f_1^{88}}+8168472576 q^4\frac{f_4^{16}f_2^{64}}{f_1^{80}}+14233370624q^5\frac{f_4^{32}f_2^{40}}{f_1^{72}}\nonumber  \\
&\quad \left. +2399141888q^6\frac{f_4^{48}f_2^{16}}{f_1^{64}}\right).
\end{align}
This proves \eqref{32n21}.

Employing \eqref{1/f1^4} to \eqref{e 16n+13}, and then extracting the terms of the form $q^{2n+1}$, we obtain
\begin{align}\label{e 32n+29}
    \sum_{n\geq 0}\overline{R_8^\ast}(32n+29)q^n&=256\left(311\frac{f_2^{154}}{f_1^{110}f_4^{44}}+223520q\frac{f_2^{130}}{f_1^{102}f_4^{28}}+21601536q^2\frac{f_2^{106}}{f_1^{94}f_4^{12}} \right. \nonumber \\ 
&\quad +486064128q^3\frac{f_4^4f_2^{82}}{f_1^{86}}+2747334656q^4\frac{f_4^{20}f_2^{58}}{f_1^{78}}+3021996032q^5\frac{f_4^{36}f_2^{34}}{f_1^{70}} \nonumber  \\
&\quad \left. +184549376q^6\frac{f_4^{52}f_2^{10}}{f_1^{62}}\right).
\end{align}
This proves \eqref{32n29}.

For the rest of the congruences, we skip the proof but mention the relevant steps below:
\begin{itemize}
    \item To obtain \eqref{32n25}, we employ \eqref{1/f1^2} in \eqref{e 16n+9},
    \item To obtain \eqref{64n53}, we employ \eqref{1/f1^4} in \eqref{e 32n+21},
    \item To obtain \eqref{64n37}, we first employ \eqref{1/f1^4} in \eqref{e 16n+5}, and \textcolor{blue}{then} employ \eqref{1/f1^4} in the resulting equation,
    \item To obtain \eqref{128n85}, we first employ \eqref{1/f1^4} in \eqref{e 32n+21}, and \textcolor{blue}{then} employ \eqref{1/f1^4} in the resulting equation,
    \item To obtain \eqref{128n117}, we first employ \eqref{1/f1^4} in \eqref{e 32n+21}, and \textcolor{blue}{then} employ \eqref{1/f1^4} in the resulting equation.
\end{itemize}
\end{proof}

\section{Proof of Theorem \ref{saik}}\label{sec-elm2}

\begin{proof}[Proof of Theorem \ref{saik}]
We recall
\[\varphi(q):=1+2\sum_{n\geq 0}q^{n^2}=\varphi(q^4) + 2q \psi (q^8),\] with $\psi(q):=\sum_{n\geq 0}q^{n(n+1)/2}$, as well as the following properties (see \cite[(1.5.8)]{power} and \cite[(1.5.16)]{power} respectively):
\begin{equation}\label{phi-1}
    \varphi(-q)=\frac{f_1^2}{f_2},
\end{equation}
and
\begin{equation}\label{phi-2}
    \frac{1}{\varphi(-q)}=\prod_{i\geq 0}\varphi(q^{2^i})^{2^i}.
\end{equation}
Now, using \eqref{phi-1} and \eqref{phi-2} we rewrite
\[
\sum_{n\geq 0}\overline{R_{2^k}^\ast}(n)q^n=f_{2^k}\prod_{i\geq 0}\varphi(q^{2^i})^{2^i}.
\]
Since $f_{2^k}\prod_{i\geq 3}\varphi(q^{2^i})^{2^i}$ is a function of $q^8$ for all $k\geq 3$, to prove our result we can ignore that part and perform a $8$-dissection of the remaining part.

As the highest modulus involved in the theorem is $16$ and the other moduli are divisors of $16$, we will prove our result if we consider this $8$-dissection modulo $16$. Rewriting
\[
\sum_{n\geq 0}\overline{R^\ast_{2^k}}(n)q^n=\left(\sum_{j=0}^{7}a_{t,j}q^{j}F_{t,j}(q^8)\right) \left(\prod_{i\geq3}\varphi(q^{2^i})\right)^{2^{i}},
\]
where $F_{t,j}(q^8)$ is a function of $q^8$ whose power series representation has integer coefficients. It suffices to just prove the congruences
\begin{align*}
    a_{t,1} &\equiv 0 \pmod{2},\\
    a_{t,2} &\equiv 0 \pmod{2},\\
    a_{t,3} &\equiv 0 \pmod{4},\\
    a_{t,4} &\equiv 0 \pmod{2},\\
    a_{t,5} &\equiv 0 \pmod{8},\\
    a_{t,6} &\equiv 0 \pmod{4},\\
    a_{t,7} &\equiv 0 \pmod{16}.
\end{align*}
However, this already follows from a result of Saikia and Sarma \cite[Theorem 1.9]{SaikiaSarma}, who follow the same pattern as the proofs of \cite[Theorem 2.2]{selbol} and \cite[Lemma 6.1]{SSS}, so we omit the details here. The interested reader can refer to the above-mentioned papers to fill in the details.
\end{proof}

\section{Proofs of Theorems \ref{thmjames1}, \ref{thmjames2}, and \ref{thmjames3}}\label{sec-elm3}

From \eqref{e1.010101}, we see that 
\begin{equation*}
        \sum_{n\geq 0}\overline{R_6^\ast}(3n+2)q^n \equiv 4f_6^3 \pmod{8}. 
    \end{equation*}
So from \eqref{e2.0.3.3}, we then know that $\overline{R_6^\ast}(3n+2) \equiv 0 \pmod{8}$ if $n$ cannot be written in the form $6k(k+1)/2$ or $3k(k+1)$ for any $k$.

 We note that $2n+1$ cannot be written as $3k(k+1)$ because $3k(k+1)$ is even for any $k$ and $2n+1$ is always odd. This proves Theorem \ref{thmjames1}. Again, we note that $3n+1$ and $3n+2$ cannot be written as $3k(k+1)$ because $3n+1$ and $3n+2$ are not divisible by 3 but $3k(k+1)$ is. This proves Theorem \ref{thmjames2}.

\begin{proof}[Proof of Theorem \ref{thmjames3}]
We ask whether $pn+r$ is expressible as $3k(k+1)$.  If it is, then this would imply that $r\equiv 3k(k+1) \pmod{p}$ which implies that 
$\text{inv}(3,p)\cdot r \equiv k(k+1) \pmod{p}$ which then implies that $\text{inv}(3,p)\cdot 4r +1 \equiv (2k+1)^2 \pmod{p}$.  But this cannot happen because of the assumption that $\text{inv}(3,p)\cdot 4r +1$ is a quadratic nonresidue modulo $p$.  This completes the proof.  
\end{proof}

\section{Proof of Theorem \ref{thm1.00}}\label{sec:mf1}

We now begin the proof of Theorem \ref{thm1.00}.

\begin{proof}[Proof of Theorem \ref{thm1.00}]
Rewriting \eqref{z1} in terms of eta quotients, we get via Lemma \ref{lm}
\begin{align*}
\sum_{n \geq 0}\overline{R_6^\ast}(18n+2)q^{8n+1} \equiv 4\eta(4z)^6 \pmod{8}.
\end{align*}	
Let $\eta(4z)^6 = \sum_{n \geq 0} a(n)q^n.$ Then $a(n) = 0$ if $n\not\equiv 1\pmod{8}$ and for all $n\geq 0$, 
\begin{align}\label{1_0}
\overline{R_6^\ast}(18n+2) \equiv a(8n+1) \pmod 8.
\end{align}	
By Theorem \ref{thm_ono1}, we have $\eta(4z)^6 \in S_3(\Gamma_0(16))$. Since $\eta(4z)^6$ is a Hecke eigenform (see, for example \cite{Martin}), \eqref{hecke1} and \eqref{hecke3} yield
\begin{align*}
	\eta(4z)^6|T_p = \sum_{n \geq 1} \left(a(pn) + \chi(p)p^2 a\left(\frac{n}{p}\right) \right)q^n = \lambda(p) \sum_{n \geq 1} a(n)q^n,
\end{align*}
which implies 
\begin{align}\label{1.12}
	a(pn) + \chi(p)p^2 a\left(\frac{n}{p}\right) = \lambda(p)a(n).
\end{align}
Putting $n=1$ and noting that $a(1)=1$, we readily obtain $a(p) = \lambda(p)$.
Since $a(p)=0$ for all $p \not\equiv 1 \pmod{8}$, we have $\lambda(p) = 0$.
From \eqref{1.12}, we obtain 
\begin{align}\label{new-1.0}
a(pn) + \chi(p)p^2 a\left(\frac{n}{p}\right) = 0.
\end{align}
From \eqref{new-1.0}, we derive that for all $n \geq 0$ and $p\nmid r$, 
\begin{align}\label{1.101}
a(p^2n + pr) = 0
\end{align}
and  
\begin{align}\label{1.310}
a(p^2n) = - \chi(p)p^2 a(n)\equiv a(n) \pmod{2}.
\end{align}
Substituting $n$ by $8n-pr+1$ in \eqref{1.101} and together with \eqref{1_0}, we find that
\begin{align}\label{1.410}
\overline{R_6^\ast}\left(18p^2n + \frac{9p^2-1}{4}+ 9pr\frac{1-p^2}{4}\right) \equiv 0 \pmod 8.
\end{align}	
Substituting $n$ by $8n+1$ in \eqref{1.310} and using \eqref{1_0}, we obtain
\begin{align}\label{1.510}
\overline{R_6^\ast}\left(18p^2n + \frac{9p^2-1}{4}\right) \equiv \overline{R_6^\ast}(18n+2) \pmod 8.
\end{align}
Since $p \geq 3$ is prime, so $4\mid (1-p^2)$ and $\gcd \left(\frac{1-p^2}{4} , p\right) = 1$.  
Hence when $r$ runs over a residue system excluding the multiple of $p$, so does $\frac{1-p^2}{4}r$. Thus \eqref{1.410} can be rewritten as
\begin{align}\label{1.610}
\overline{R_6^\ast}\left(18p^2n + \frac{9p^2-1}{4}+ 9pj\right) \equiv 0 \pmod 8,
\end{align}
where $p \nmid j$. 
\par Now, $p_i \geq 3$ are primes such that $p_i \not\equiv 1 \pmod 8$. Since 
\begin{align*}
18p_1^2\dots p_{k}^2n + \frac{9p_1^2\dots p_{k}^2-1}{4}=18p_1^2\left(p_2^2\dots p_{k}^2n + \frac{p_2^2\dots p_{k}^2-1}{8}\right)+\frac{9p_1^2-1}{4},
\end{align*}
using \eqref{1.510} repeatedly we obtain that
\begin{align*}
\overline{R_6^\ast}\left(18p_1^2\dots p_{k}^2n + \frac{9p_1^2\dots p_{k}^2-1}{4}\right) \equiv \overline{R_6^\ast}(18n+2) \pmod{8},
\end{align*}
which implies
\begin{equation}\label{1.710}
\overline{R_6^\ast}\left( p_1^2 \cdots p_k^2 n + \frac{p_1^2\cdots p_k^2 -1}{4} \right) \equiv \overline{R_6^\ast}(n) \pmod{8}.
\end{equation}
Let $j\not\equiv 0\pmod{p_{k+1}}$. Then \eqref{1.610} and \eqref{1.710} yield
\begin{align*}
\overline{R_6^\ast}\left(18p_1^2\dots p_{k+1}^2n + \frac{9p_1^2\dots p_{k}^2p_{k+1}(4j+p_{k+1})-1}{4}\right) \equiv 0 \pmod{8}.
\end{align*}
This completes the proof of the theorem.
\end{proof}

\section{Proof of Theorem \ref{t4}}\label{sec:mf2}

\begin{proof}
Magnifying $q \to q^8$ and multiplying $q$ on both sides of $\eqref{z1}$, we have
\begin{equation}\label{e4}
    \sum_{n \geq 0}\overline{R_6^\ast}(18n+2)q^{8n+1} \equiv 4qf_8^3 = 4\Delta(z) = 4\sum_{n \geq 1}\tau(n)q^n \pmod{8}.
\end{equation}
Hence, from $\eqref{e2.0.3.3}$, we have
\begin{equation*}
    \Delta(z)=qf_1^{24} \equiv qf_8^3 \equiv \sum_{n \geq 0}q^{(2n+1)^2} \pmod{2}.
\end{equation*}
Therefore, we have
\begin{equation}\label{e5}
    \sum_{n \geq 0}\overline{R_6^\ast}(18n+2)q^{8n+1}\equiv 4\sum_{n \geq 0}q^{(2n+1)^2} \pmod{8}.
\end{equation}
If $s\equiv1 \pmod{8}$ and $\left(\dfrac{s}{p}\right) = -1$, then, for any $n\geq 0$, $8np+s$ cannot be an odd square. This implies that the coefficients of $q^{8np+s}$ in the left-hand side of $\eqref{e5}$ must be even. It follows that
\begin{equation}\label{e6}
    \overline{R_6^\ast}\left( 18pn + \dfrac{9s-1}{4} \right) \equiv 0 \pmod{8}.
\end{equation}
This proves \eqref{e12.0.1}.

Since $\tau(p) \equiv 0 \pmod{2}$ and $\Delta(n)$ is a Hecke eigenform, we have
\begin{equation*}
    \Delta(z)\mid T_{p} = \tau(p)\Delta(z) \equiv 0 \pmod{2}.
\end{equation*}
By $\eqref{hecke1}$ and $\eqref{e4}$, we get
\begin{equation*}
    \sum_{n \geq 0}\overline{R_6^\ast}(18n+2) q^{8n+1} \mid T_{p} \equiv \sum_{n \geq n_0}\left( \overline{R_6^\ast}\left( \dfrac{9pn-1}{4} \right) + \overline{R_6^\ast}\left( \dfrac{9n/p-1}{4}\right) \right)q^n \equiv 0 \pmod{8}.
\end{equation*}
If we write $m=\dfrac{9n/p-1}{4}\in N $, then $\dfrac{9pn-1}{4} = p^2m + \dfrac{p^2-1}{4}$. So, we have
\begin{equation*}
    \overline{R_6^\ast}\left( p^2m + \dfrac{p^2-1}{4} \right) + \overline{R_6^\ast}(m) \equiv 0 \pmod{8}.
\end{equation*}
That is, 
\begin{equation*}
   \overline{R_6^\ast}\left( p^2m + \dfrac{p^2-1}{8} \right) \equiv \overline{R_6^\ast}(m) \pmod{8}.
\end{equation*}
By induction, for $k\geq 1$ we find that
\begin{equation*}
    \overline{R_6^\ast}\left( p^{2k}m + \dfrac{p^{2k}-1}{4} \right) \equiv \overline{R_6^\ast}(m) \pmod{8}.
\end{equation*}
Considering $\eqref{e6}$, we get
\begin{equation*}
   \overline{R_6^\ast}\left( p^{2k+1}n+\dfrac{sp^{2k}-1}{4} \right) \equiv \overline{R_6^\ast}\left( pn + \dfrac{s-1}{4} \right) \equiv 0 \pmod{8}.
\end{equation*}
This proves \eqref{e12.0.2}.
\end{proof}

\section{Concluding Remarks}\label{sec:concl}

\begin{enumerate}
    \item Congruences \eqref{mun-5} and \eqref{mun-6} are still in need of elementary proofs.
    \item The interested reader can think of extending the list of congruences proved in Theorem \ref{elthm} using other dissection formulas available in the literature. Some other congruences have been studied recently by N. Saikia and Paksok \cite{SaikiaPaksok}.
\item Can we generalize Theorem \ref{thm:nathsel} to find an infinite family of congruences?
\item From the examples in Theorem \ref{t4}, we observe that the congruence modulo $8$ appears to extend to modulo $128$. We state this as a conjecture below.

\begin{conjecture}
    For all $n \geq 0$, we have
    \begin{equation*}
        \overline{R_6^\ast}\left(54n + 38\right) \equiv 0 \pmod{128}.
    \end{equation*}
\end{conjecture}

\begin{conjecture}
    For all $n \geq 0$ and $k \geq 1$, we have
    \begin{equation*}
        \overline{R_6^\ast}\left(18 \cdot 3^{2k+1} n + \dfrac{153 \cdot 3^{2k} - 1}{4}\right) \equiv 0 \pmod{128}.
    \end{equation*}
\end{conjecture}

\end{enumerate}

\subsection*{Acknowledgements}

The authors are grateful to the anonymous referees for their comments.

\subsection*{Declarations} The authors make the following declaration.
\begin{enumerate}
    \item The authors declare no conflicts of interest.
    \item The authors have no relevant financial or non-financial interests to disclose.
    \item All authors contributed to the study conception and design. All authors read and approved the final manuscript.
\end{enumerate}

\end{document}